%% file: main.tex
\documentclass{amsart}
\usepackage[utf8]{inputenc}
\usepackage{amsmath}
\usepackage{amssymb}
\usepackage{caption}
\usepackage{amsthm}
\usepackage[hidelinks]{hyperref}
\usepackage{cleveref}
\usepackage{hyperref}
\usepackage{nicefrac}
\usepackage[inline]{enumitem}
\usepackage{enumitem}
\usepackage{comment}
\usepackage{todonotes}

\usepackage{amssymb}
\usepackage{mathrsfs}

\usepackage{array,float}

\usepackage{tikz}
\usetikzlibrary{calc}
\usetikzlibrary{arrows.meta}
\usetikzlibrary{shapes,arrows}
\usetikzlibrary{fit,positioning}
\usetikzlibrary{patterns,decorations.pathreplacing}
\usepackage{xkeyval}
\usepackage{moreverb}
\usepackage{epic}
\usepgfmodule{shapes,plot,decorations}

\usepackage{booktabs} 
\usepackage[normalem]{ulem}

\newcommand{\HH}{\mathbb{H}}

\newcommand{\PO}{\mathbf{PO}}

\usepackage[warn]{mathtext}

\newcommand{\id}{\mathrm{Id}}

\newtheorem{theo}{Theorem}[section]
\newtheorem{cor}{Corollary}[theo]
\newtheorem{lem}[theo]{Lemma}
\newtheorem{prop}[theo]{Proposition}

\newtheorem{question}[theo]{Question}

\newtheorem{assumption}{Assumption}

\theoremstyle{remark}
\newtheorem{remark}[theo]{Remark}

\theoremstyle{remark}

\title[Commensurability classes of Coxeter polyhedra in $\mathbb{H}^4$ and $\mathbb{H}^5$]{Infinitely many commensurability classes of compact Coxeter polyhedra in $\mathbb{H}^4$ and $\mathbb{H}^5$}

\author{Nikolay Bogachev}
\address{Department of Computer and Mathematical Sciences, University of Toronto Scarborough, 1095 Military Trail, Toronto, ON M1C 1A4, Canada}
\email{n.bogachev@utoronto.ca}

\author{Sami Douba}
\address{Institut des Hautes \'Etudes Scientifiques, 
Universit\'e Paris-Saclay, 35 route de Chartres, 91440 Bures-sur-Yvette, France}
\email{douba@ihes.fr}

\author{Jean Raimbault}
\address{Institut de Mathématiques de Marseille, UMR 7373, CNRS, Aix-Marseille Université}
\email{jean.raimbault@univ-amu.fr}

\input{macros}

\numberwithin{equation}{section}

\begin{document}

\maketitle

\begin{abstract}
We prove that certain families of compact Coxeter polyhedra in 4- and 5-dimensional hyperbolic space constructed by Makarov give rise to infinitely many commensurability classes of reflection groups in these dimensions. 
\end{abstract}


\section{Introduction}

Among the most attractive and concrete examples of discrete groups of isometries of hyperbolic space $\mathbb{H}^d$ are the hyperbolic reflection groups, i.e., those generated by reflections. That being said, the existence of {\em cocompact} such groups---in other words, the existence of compact hyperbolic Coxeter polyhedra---is ultimately a low-dimensional phenomenon. Indeed, Vinberg \cite{vinberg81nonexistence} showed that there are no such polyhedra in dimensions $d \geq 30$, and the highest dimension in which such a polyhedron has been exhibited is $8$; see Bugaenko \cite{bugaenko84, bugaenko92}.

We may nevertheless ask whether these polyhedra abound in the dimensions where they do exist. It is clear, for example, that a compact right-angled  polyhedron in $\mathbb{H}^d$ gives rise to infinitely many such polyhedra by successively doubling along walls. While compact right-angled polyhedra cease to exist in $\mathbb{H}^d$ as soon as $d > 4$ as follows from the Nikulin inequality \cite{nikulin81} (see \cite{potyagailovinberg}), this doubling method can be applied to a more general family of polyhedra than those that are right-angled, an observation that was exploited by Allcock \cite{All06} to construct infinitely many compact Coxeter polyhedra in $\mathbb{H}^d$ for each $d \leq 6$.
However, this method falls short of addressing the following question (see \cite{164099}).

\begin{question}\label{question}
Given $d <30$, are there infinitely many {\em commensurability classes} of compact Coxeter polyhedra in $\HH^d$?
\end{question}

\noindent Here, we say two compact Coxeter polyhedra in $\mathbb{H}^d$ are {\em commensurable} if the corresponding reflection groups $\Gamma_1, \Gamma_2 < \mathrm{Isom}(\mathbb{H}^d) \cong \PO(d,1)$ are commensurable (in the wide sense), that is, if there is some $g \in \PO(d,1)$ such that $\Gamma_1 \cap g \Gamma_2 g^{-1}$ has finite index in both $\Gamma_1$ and $g\Gamma_2 g^{-1}$. We will routinely conflate hyperbolic Coxeter polyhedra with their associated reflection groups.

The dimensions $d$ for which Question \ref{question} is known to have a positive answer include those in which one can find a sequence of compact hyperbolic Coxeter polyhedra with arbitrarily small dihedral angles, since the (adjoint) trace fields of such polyhedra (which are commensurability invariants) have unbounded degree. Such sequences are known to exist for $d=2$ by the Gauss--Bonnet theorem, and for $d=3$ by Andreev's theorem \cite{andreev1970, rhd07}, but their existence in higher dimensions is open. Note that there can even be uncountably many commensurability classes of compact hyperbolic Coxeter {\em polygons} of a fixed combinatorial type, whereas the combinatorics of a compact Coxeter polyhedron in $\mathbb{H}^d$ determine its geometry for $d \geq 3$ by Mostow rigidity.

A natural way (which may be viewed as a generalization of the doubling method) to try to build new polyhedra is to glue copies of smaller polyhedra along isometric facets, forming so-called {\em garlands}. For  $d=4,5$, Makarov \cite{Mak68} discovered compact Coxeter polyhedra $P_i^d \subset \mathbb{H}^d$, $i=1,2$, whose garlands yield infinitely many compact Coxeter polyhedra in $\mathbb{H}^d$. Felikson and Tumarkin \cite{feliksontumarkin2014} have suggested that Makarov's garlands in dimensions $4$ and $5$ constitute  infinite families even up to commensurability. We prove that this is indeed the case, and hence settle Question~\ref{question} in the affirmative for $d=4,5$. 

\begin{theo} \label{Main}
  There are infinitely many commensurability classes of compact Coxeter polyhedra in $\HH^d$ for $d=4,5$.
\end{theo}

\noindent Indeed, by applying techniques of the last author \cite{croissmax} to Makarov's garlands in dimension $4$, we obtain the following quantitative result (compare with \cite[Thm.~1.2]{All06} and \cite[Cor.~1.2]{croissmax}).

\begin{theo}\label{quantitative}
There is some $c>1$ such that, for $V$ sufficiently large, the number of pairwise incommensurable compact Coxeter polyhedra in $\mathbb{H}^4$ of volume $\leq V$ is at least $c^V$.
\end{theo}

In view of Theorem \ref{Main}, the only dimension $d$ in which infinitely many compact hyperbolic Coxeter polyhedra have been exhibited but for which Question \ref{question} remains open is $d=6$. We remark that, while in any infinite family of pairwise incommensurable compact hyperbolic Coxeter polyhedra only finitely many members are arithmetic \cite{ABSW,Nikulin,Fisher_Hurtado,coxeter_finiteness}, not a single nonarithmetic compact Coxeter polyhedron in $\mathbb{H}^6$ is known, as pointed out by Vinberg \cite{vinberg2014}.

When we set out to prove Theorem \ref{Main}, our initial hope was that the polyhedra $P_1^d$ and $P_2^d$ were arithmetic and incommensurable, so that we could directly apply arguments of the last author \cite{croissmax}, which in turn build on techniques used by Gromov and Piatetski-Shapiro \cite{GPS} to verify nonarithmeticity of their ``hybrid" lattices. While this is indeed the case for $d=5$, it turns out that neither of the~$P_i^4$ are arithmetic. The $P_i^d$ are nevertheless all {\em quasi-arithmetic} (see Section \ref{QAlattices} for the definition, as well as \cite{BK22} for a broader discussion on quasi-arithmetic reflection groups). Arguments as in \cite{croissmax} then go through once we have verified that $P_1^d$ and~$P_2^d$ are not only incommensurable, but that their ambient groups are  distinct; for $d=4$, the latter, as well as proper quasi-arithmeticity of the $P_i^4$, were in fact already established by Dotti \cite[Example~3.16]{dotti}. (Note that to say two lattices $\Gamma_1, \Gamma_2 < \PO(d,1)$ have distinct ambient groups is the same as saying that the $\Gamma_i$ are incommensurable {\em in the case that the lattices are arithmetic}.) By distinguishing the ambient groups of the $P_i^d$, we can also conclude the following (compare with Thomson \cite[Thm.~1.6]{Thomson}, cf. also \cite[Thm.~1.1]{BGV23}).

\begin{theo} \label{non-quasi-arit}
  A Makarov garland in $\mathbb{H}^d$, $d=4,5$, containing both a $P_1^d$ piece and a $P_2^d$ piece is {\em not} quasi-arithmetic, so that we may take all the polyhedra in Theorem \ref{Main} (and Theorem \ref{quantitative}) to be non-quasi-arithmetic.
\end{theo}

We remark that it is not known for any $3 \leq d < 30$ whether there are infinitely many commensurability classes of quasi-arithmetic Coxeter polyhedra in~$\mathbb{H}^d$. On the other hand, and in contrast to the arithmetic setting, Dotti and Kolpakov~\cite{dottikolpakov2022} exhibited infinitely many commensurability classes of quasi-arithmetic compact hyperbolic Coxeter {\em polygons}, and indeed, a family of such polygons the degrees of whose trace fields are unbounded. It follows from work of Mila \cite{Mila} that there are only finitely many possibilities for the trace fields of Makarov's garlands, so that we do not address the following variant of Question \ref{question}.

\begin{question}\label{question_tracefield}
Given $4 \leq d <30$, are there infinitely many options for the trace fields of compact Coxeter polyhedra in $\HH^d$?
\end{question}

\noindent As observed in \cite{dottikolpakov2022}, it follows from work of Vinberg \cite{vinberg84absence} that there are only finitely many options for the trace fields of compact quasi-arithmetic Coxeter polyhedra (if any) in $\mathbb{H}^d$ for $d \geq 14$.


\subsection*{Structure of the paper} In Section \ref{QAlattices} we review trace fields and ambient groups for discrete groups of isometries and prove some facts about quasi-arithmetic lattices. We describe Makarov's polyhedra in Section \ref{desc_polytopes} and compute their trace fields and ambient groups there. In Section \ref{proofs} we deduce Theorems \ref{Main},  \ref{quantitative}, and  \ref{non-quasi-arit}. 


\subsection*{Acknowledgments} We thank Pavel Tumarkin for pointing out to us Dotti's example \cite[Example~3.16]{dotti}. We are  grateful to Daniel Allcock and Pierre Py for pointing out misprints in a previous draft of this article. We also thank the referee for their reading of this article and for their comments. S.~D. was supported by the Huawei Young Talents Program. J.~R. was supported by grant AGDE - ANR-20-CE40-0010-01.


\section{Quasi-arithmetic lattices}\label{QAlattices}

\subsection{Definitions}

We very briefly review the definitions and facts we will need about quasi-arithmetic lattices. See \cite[Section 2]{BBKS} for a more in-depth introduction to the topic. 

Let $G$ be a noncompact semisimple real Lie group and $k \subset \RR$ a totally real number field. A $k$-group $\G$ is said to be {\em admissible} for $G$ if $\G(\RR)$ is isogenous to $G$ and $\G(k \otimes_\sigma \RR)$ is compact for any non-identity embedding $\sigma :\: k \to \RR$.

A lattice $\Gamma$ in $G$ is said to be {\em quasi-arithmetic} if there exist $k$ and $\G$ with $\G$ admissible for $G$ and $\Gamma \cap \G(k)$ a lattice in $G$. If one can find such $k$ and $\G$ such that moreover $\Gamma \cap \G(\mathcal{O}_k)$ is a lattice in $G$, i.e., $\Gamma$ is commensurable to $\G(\mathcal{O}_k)$, where $\mathcal{O}_k$ is the ring of integers of $k$, then $\Gamma$ is {\em arithmetic}. If $\Gamma$ is quasi-arithmetic but not arithmetic, one says $\Gamma$ is {\em properly quasi-arithmetic}. 

Another characterization of quasi-arithmeticity is as follows. Let $\Gamma$ be a lattice in~$G$, and $k(\Gamma)$ the subfield of $\RR$ generated by the traces $\tr(\ad g)$ for $g \in \Gamma$. There is (up to $k(\Gamma)$-isogeny) a unique $k(\Gamma)$-group $\G_\Gamma$ such that $\G_\Gamma(\mathbb{R})$ is~isogenous to~$G$ and~$\Gamma$ is virtually contained in $\G_\Gamma(k(\Gamma))$ via this isogeny. The field $k(\Gamma)$ and the group~$\G_\Gamma$ are the {\em adjoint trace field} and the {\em ambient group} of $\Gamma$, respectively. (See~\cite{Vinberg_trace} for the existence and unicity of these invariants, the so-called {\em Vinberg invariants} of $\Gamma$.) Then $\Gamma$ is quasi-arithmetic if and only if $k(\Gamma)$ is a totally real number field and~$\G_\Gamma$ is admissible for $G$. If $\Gamma$ is quasi-arithmetic, then $\Gamma$ is arithmetic if and only if  $\tr(\ad g)$ is an algebraic integer for each $g \in \Gamma$.

We remark, although this will not be used in the sequel, that it follows from well-known results that the semisimple Lie groups containing irreducible properly quasi-arithmetic lattices are (up to isogeny) exactly the groups $\PO(d, 1)$. Indeed, for higher-rank groups and $\mathrm{PSp}(d, 1)$, $\mathrm{F}_4^{-20}$ all irreducible lattices are arithmetic by Margulis \cite[Ch.~IX]{margulis91discrete} and Corlette and Gromov--Schoen \cite{corlette, gromovschoen}, and for lattices in $\mathrm{PU}(d,1)$ ($d\ge 2$) traces are always integral (as proven in \cite[Theorem 1.3.1]{baldiullmo} or \cite[Theorem 1.5]{bfms} using work of Esnault--Groechenig \cite[Thm.~1.1]{EG2018}). On the other hand, as observed by Thomson~\cite{Thomson}, it follows from work of Belolipetsky--Thomson \cite{belolipetskythomson}, or a result of Bergeron--Haglund--Wise \cite{bergeronhaglundwise}, that a construction of Agol \cite{agol2006systoles} yields properly quasi-arithmetic lattices in $\PO(d,1)$ for each $d \geq 2$.


\subsection{Algebraic rigidity of quasi-arithmetic lattices}

The following result is similar to \cite[Proposition 3.2]{Thomson}.  

\begin{prop}\label{similartothomson}
Let $G$ be a semisimple Lie group and suppose for some subfield $k \subset \mathbb{R}$ one has $k$-groups ${\bf G}_1, {\bf G}_2$ and Lie group isomorphisms $\rho_i :{\bf G}_i(\mathbb{R}) \rightarrow G$, $i= 1,2$, such that $\rho_1({\bf G}_1(k)) \cap \rho_2({\bf G}_2(k))$ is Zariski-dense in~$G$. Then the ${\bf G}_i$ are $k$-isogenous.
\end{prop}

\begin{proof}
Let $\Lambda = \rho_1({\bf G}_1(k)) \cap \rho_2({\bf G}_2(k))$ and $\mathrm{Ad}: G \rightarrow \mathrm{GL}(\mathfrak{g})$ be the adjoint representation of $G$, where $\mathfrak{g}$ denotes the Lie algebra of $G$. The adjoint trace field of~$\Lambda$ is contained in $k$ 
 since $\rho_1^{-1}(\Lambda) \subset {\bf G}_1(k)$. By \cite[Thm.~1]{Vinberg_trace}, it follows that there is a basis for $\mathfrak{g}$ with respect to which $\mathrm{Ad}(\Lambda)$ has entries in $k$. Identify $\mathrm{GL}(\mathfrak{g})$ with $\mathrm{GL}_n(\mathbb{R})$ via this basis, where $n= \mathrm{dim}(\mathfrak{g})$. 

 Since $\Lambda$ is assumed to be Zariski-dense in $G$, we have that $\mathrm{Ad}(G) \subset {\bf G}(\mathbb{R})$, where~${\bf G}$ is the $k$-closure of $\mathrm{Ad}(\Lambda)$. Now each $\mathrm{Ad}\circ \rho_i : {\bf G}_i \rightarrow {\bf G}$ is an $\mathbb{R}$-isogeny mapping the Zariski-dense subgroup $\rho_i^{-1}(\Lambda)$ of ${\bf G}_i(k)$ into ${\bf G}(k)$. It follows that $\mathrm{Ad}\circ \rho_i$ is moreover a $k$-isogeny for $i=1,2$; see, for instance, \cite[Prop.~3.1.10]{Zimmer_book}.
\end{proof}

This implies the following generalisation of \cite[1.6]{GPS} to quasi-arithmetic lattices. 

\begin{cor} \label{qa_hybrid}
  If $\Gamma_1$ and $\Gamma_2$ are two quasi-arithmetic lattices in $G$ and $\Gamma_1 \cap \Gamma_2$ is Zariski-dense in $G$ then the ambient groups of $\Gamma_1$ and $\Gamma_2$ coincide. 
\end{cor}

\begin{proof}
  To prove Corollary \ref{qa_hybrid} we need only  prove that the adjoint trace fields of $\Gamma_1$ and $\Gamma_2$ are equal to some fixed field $k$; indeed, the ambient groups ${\bf G}_i$ of the $\Gamma_i$ would then be $k$-groups satisfying the assumptions of Proposition \ref{similartothomson}. To that end, let $\Lambda = \Gamma_1 \cap \Gamma_2$ and $\ell$ be the adjoint trace field of $\Lambda$; we prove that the adjoint trace field $k$ of $\Gamma_1$ is equal to $\ell$. 

  We have $\ell \subset k$. If $\ell$ is a proper subfield of $k$ then
  by the Galois correspondence (applied to some Galois number field containing $k$) there is an embedding $\sigma \not= \id$ of $k$ into $\RR$ such that $\sigma|_\ell = \id$. It follows that 
  $\ad(\Lambda) = \sigma(\ad(\Lambda)) \subset \sigma(\ad(\G_1(k))$, and the latter is contained in the compact group $\ad(\G_1(k \otimes_\sigma \RR))$; we recall that here $\G_1$ denotes the ambient group of $\Gamma_1$. We conclude that $\mathrm{Ad}(\Lambda)$ is precompact in $\mathrm{Ad}(G)$, but this contradicts Zariski-density of $\Lambda$ in $G$. 
\end{proof}


\section{Makarov's polyhedra} \label{desc_polytopes}

\subsection{Vinberg invariants for hyperbolic reflection groups}

Let $P$ be a Coxeter polyhedron in $\HH^d$ and let $H_1, \ldots, H_N$ be the hyperplanes in $\HH^d$ supporting the facets, i.e., codimension-1 faces, of $P$. In the hyperboloid model for $\HH^d$ each $H_i$ is associated with a subspace of signature $(d-1, 1)$ in the Lorentz space $\RR^{d, 1}$; the orthogonal to this subspace is a positive line, so we can choose a unit vector $e_i$ orthogonal to this hyperplane.\footnote{The sign of $e_i$ is not important for us.}

The Gram matrix $G(P)$ for $P$ is then the Gram matrix of the vectors $e_i$, that is,
\[
G(P) = \left( (e_i, e_j)_{\RR^{d, 1}} \right)_{1 \le i, j \le N}. 
\]

Let $k(P)$ be the field generated by all cyclic products of the matrix $G(P)$ (a cyclic product of a matrix $(a_{ij})_{1\le i, j \le N}$ is a product $a_{i_1i_2}a_{i_2i_3}\cdots a_{i_ki_1}$). In general $k(P)$ is a (possibly proper) subfield of the field generated by the entries of $G(P)$.

If $P$ is a Coxeter polyhedron in $\HH^d$ we denote by $\Gamma_P$ the associated reflection group. The following statement follows from work of Vinberg, see 
\cite[Section 4, Theorem 5]{Vinberg_trace}.

\begin{theo}[Vinberg] \label{V} 
Let $P$ be a Coxeter polyhedron in $\HH^d$ and suppose $\Gamma_P$ is Zariski-dense in $\PO(d,1)$. Then the adjoint trace field of $\Gamma_P$ is $k(P)$. 
\end{theo}


The following criterion is useful for determining when two orthogonal groups are isomorphic (see \cite[2.6]{GPS}), and we will use it to check the conditions in Proposition~\ref{commclass_garlands}. 

\begin{lem} \label{isogeny}
  Let $k$ be a field of characteristic $0$, $m \ge 2$, and $Q, Q'$ be two non-degenerate quadratic forms on $k^m$. Then $\PO(Q)$ is $k$-isogenous to $\PO(Q')$ if and only if $Q, Q'$ are similar over $k$, that is, if and only if there is some $\lambda \in k^\times$ such that the forms $Q$ and $\lambda Q'$ are isometric over $k$. 
\end{lem}


\subsection{In four dimensions} \label{4d_Makarov}

\subsubsection{Description}
We follow Vinberg \cite[p.62]{Vinberg_polytopes}. Consider the abstract Coxeter simplex given by the following diagram:
\begin{center}
\begin{tikzpicture}
  \coordinate (1) at (0, 0) ;
  \coordinate (2) at (1, 0) ;
  \coordinate (3) at (2, 0) ;
  \coordinate (4) at (3, 0) ;
  \coordinate (5) at (4, 0) ;

  \draw (0,0.06) -- (1,0.06) ; 
  \draw (0,-0.06) -- (1,-0.06) ;
  \draw (2) -- (3) ; 
  
  \draw (3) -- (4) ; 
  \draw (2,0.07) -- (3,0.07) ; 
  \draw (2,-0.07) -- (3,-0.07) ;
  
  \draw (4) -- (5) ;

  \fill[white] (1) circle (2pt) ;
  \fill[white] (2) circle (2pt) ; 
  \fill[white] (3) circle (2pt) ;
  \fill[white] (4) circle (2pt) ;
  \fill[white] (5) circle (2pt) ;
  
  \draw (1) circle (2pt) ;
  \draw (2) circle (2pt) ; 
  \draw (3) circle (2pt) ;
  \draw (4) circle (2pt) ;
  \draw (5) circle (2pt) ; 
\end{tikzpicture}
\end{center}
The Gram matrix has signature $(4,1)$ so that this abstract simplex can be realized as a projective simplex $S_1^4 \subset \mathbb P(\RR^{4,1})$. Three of its vertices have spherical link so they lie within $\HH^4$. The remaining two, whose opposite facets correspond to the right- and leftmost vertices in the diagram, have links corresponding to the diagrams
\begin{center}
\begin{tikzpicture}
  \coordinate (1) at (0, 0) ;
  \coordinate (2) at (1, 0) ;
  \coordinate (3) at (2, 0) ;
  \coordinate (4) at (3, 0) ;

  \draw (0,0.06) -- (1,0.06) ; 
  \draw (0,-0.06) -- (1,-0.06) ;
  
  \draw (2) -- (3) ; 
  \draw (3) -- (4) ; 
  \draw (2,0.07) -- (3,0.07) ; 
  \draw (2,-0.07) -- (3,-0.07) ;

  \fill[white] (1) circle (2pt) ;
  \fill[white] (2) circle (2pt) ; 
  \fill[white] (3) circle (2pt) ;
  \fill[white] (4) circle (2pt) ;
  
  \draw (1) circle (2pt) ;
  \draw (2) circle (2pt) ; 
  \draw (3) circle (2pt) ;
  \draw (4) circle (2pt) ;

  \coordinate (2p) at (6, 0) ;
  \coordinate (3p) at (7, 0) ;
  \coordinate (4p) at (8, 0) ;
  \coordinate (5p) at (9, 0) ;

  \draw (2p) -- (3p) ; 
  
  \draw (7,+0.07) -- (8,+0.07) ;
  \draw (7,0) -- (8,0) ;
  \draw (7,-0.07) -- (8,-0.07) ;
  
  \draw (4p) -- (5p) ; 

  \fill[white] (2p) circle (2pt) ; 
  \fill[white] (3p) circle (2pt) ;
  \fill[white] (4p) circle (2pt) ;
  \fill[white] (5p) circle (2pt) ;
  
  \draw (2p) circle (2pt) ; 
  \draw (3p) circle (2pt) ;
  \draw (4p) circle (2pt) ;
  \draw (5p) circle (2pt) ; 
\end{tikzpicture}
\end{center}
which represent compact simplices $T_1, T_2$ in $\HH^3$, respectively; it follows that these vertices are hyperideal and that there are two hyperplanes $H_1, H_2$ truncating $S_1^4$ into a compact polyhedron $\overline{S_1^4} \subset \HH^4$, two facets of which are orthogonal to all facets meeting them and are isometric respectively to $T_1$ and $T_2$. We let $P_1^4$ be the double of $\overline{S_1^4}$ along the $T_1$ facet. 

\medskip

The polyhedron $P_2^4$ is constructed in a similar manner: consider the Coxeter diagram
\begin{center}
\begin{tikzpicture}
  \coordinate (1) at (0, 0) ;
  \coordinate (2) at (1, 0) ;
  \coordinate (3) at (2, 0) ;
  \coordinate (4) at (3, 0) ;
  \coordinate (5) at (4, 0) ;

  \draw (1) -- (2) ;
  \draw (0,0.07) -- (1,0.07) ; 
  \draw (0,-0.07) -- (1,-0.07) ;
  \draw (2) -- (3) ; 
  \draw (3) -- (4) ; 
  \draw (2,0.07) -- (3,0.07) ; 
  \draw (2,-0.07) -- (3,-0.07) ;
  \draw (4) -- (5) ;

  \fill[white] (1) circle (2pt) ;
  \fill[white] (2) circle (2pt) ; 
  \fill[white] (3) circle (2pt) ;
  \fill[white] (4) circle (2pt) ;
  \fill[white] (5) circle (2pt) ;
  
  \draw (1) circle (2pt) ;
  \draw (2) circle (2pt) ; 
  \draw (3) circle (2pt) ;
  \draw (4) circle (2pt) ;
  \draw (5) circle (2pt) ; 
\end{tikzpicture}
\end{center}
which likewise represents a simplex $S_2^4$ in $\mathbb P(\RR^{4,1})$. As before, three of the vertices have spherical link, and the remaining two have links with diagrams
\begin{center}
\begin{tikzpicture}
  \coordinate (1) at (0, 0) ;
  \coordinate (2) at (1, 0) ;
  \coordinate (3) at (2, 0) ;
  \coordinate (4) at (3, 0) ;

  \draw (1) -- (2) ;
  \draw (0,0.07) -- (1,0.07) ; 
  \draw (0,-0.07) -- (1,-0.07) ;
  \draw (2) -- (3) ; 
  \draw (3) -- (4) ; 
  \draw (2,0.07) -- (3,0.07) ; 
  \draw (2,-0.07) -- (3,-0.07) ;

  \fill[white] (1) circle (2pt) ;
  \fill[white] (2) circle (2pt) ; 
  \fill[white] (3) circle (2pt) ;
  \fill[white] (4) circle (2pt) ;
  
  \draw (1) circle (2pt) ;
  \draw (2) circle (2pt) ; 
  \draw (3) circle (2pt) ;
  \draw (4) circle (2pt) ;

  \coordinate (2p) at (6, 0) ;
  \coordinate (3p) at (7, 0) ;
  \coordinate (4p) at (8, 0) ;
  \coordinate (5p) at (9, 0) ;

  \draw (2p) -- (3p) ; 
  \draw (7,+0.07) -- (8,+0.07) ;
  \draw (7,0) -- (8,0) ;
  \draw (7,-0.07) -- (8,-0.07) ;
  
  \draw (4p) -- (5p) ; 

  \fill[white] (2p) circle (2pt) ; 
  \fill[white] (3p) circle (2pt) ;
  \fill[white] (4p) circle (2pt) ;
  \fill[white] (5p) circle (2pt) ;
  
  \draw (2p) circle (2pt) ; 
  \draw (3p) circle (2pt) ;
  \draw (4p) circle (2pt) ;
  \draw (5p) circle (2pt) ; 
\end{tikzpicture}
\end{center}
representing compact simplices $T_3, T_2$ in $\HH^3$, respectively. Using the same cutting/doubling (this time along $T_3$), we obtain $P_2^4$.

\medskip

To summarize, we have described compact polyhedra $P_1^4, P_2^4$ in $\HH^4$ such that each has two nonadjacent facets isometric to the hyperbolic 3-simplex $T_2$ and orthogonal to all adjacent facets. The Coxeter diagrams of $P_1^4$ and $P_2^4$ are depicted in Figure~\ref{fig:p12-4}.

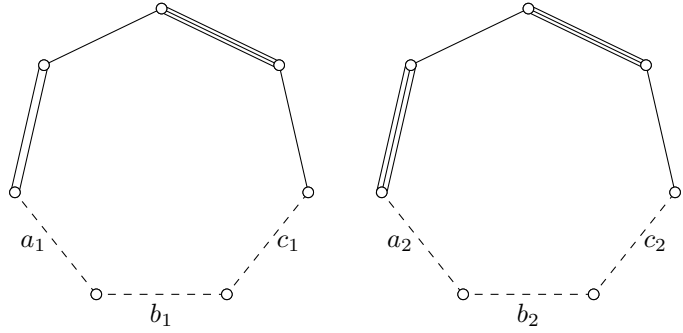
\begin{figure}
    \centering
    \begin{tikzpicture}
 \begin{scope}[rotate=90, yscale=-1] 
    \foreach \i in {1,...,7} {
      \draw[fill=white] ({360/7*(\i - 3)}:2) circle (2pt);
      \coordinate (p\i) at ({360/7*(\i - 3)}:2);
    }

    \path let \p1 = ($(p2)-(p1)$), \n1 = {atan2(\y1,\x1)} in
      coordinate (parallelStart1) at ($(p1) + (90-\n1:0.06cm)$)
      coordinate (parallelEnd1) at ($(p2) + (90-\n1:0.06cm)$)
      coordinate (parallelStart2) at ($(p1) - (90-\n1:0.06cm)$)
      coordinate (parallelEnd2) at ($(p2) - (90-\n1:0.06cm)$);

    \path let \p1 = ($(p4)-(p3)$), \n1 = {atan2(\y1,\x1)} in
      coordinate (q1) at ($(p3) + (90-\n1:0.075cm)$)
      coordinate (q2) at ($(p4) + (90-\n1:0.075cm)$)
      coordinate (s1) at ($(p3) - (90-\n1:0.075cm)$)
      coordinate (s2) at ($(p4) - (90-\n1:0.075cm)$);


    \draw (parallelStart1) -- (parallelEnd1);
    \draw (parallelStart2) -- (parallelEnd2);

    \draw (p2) -- (p3);

    \draw (q1) --  (q2);
    \draw (p3) --  (p4);
    \draw (s1) --  (s2);    
    
    \draw (p4) -- (p5);
    \draw[dashed] (p5) -- node [right] {$c_1$}  (p6);
    \draw[dashed] (p6) -- node [below] {$b_1$}  (p7);
    \draw[dashed] (p7) -- node [left] {$a_1$}  (p1);

    \foreach \i in {1,...,7} {\draw[fill=white] ({360/7*(\i - 3)}:2) circle (2pt);}
  \end{scope}
\end{tikzpicture} \quad \quad
\begin{tikzpicture}
 \begin{scope}[rotate=90, yscale=-1] 
    \foreach \i in {1,...,7} {
      \draw[fill=white] ({360/7*(\i - 3)}:2) circle (2pt);
      \coordinate (p\i) at ({360/7*(\i - 3)}:2);
    }

    \path let \p1 = ($(p2)-(p1)$), \n1 = {atan2(\y1,\x1)} in
      coordinate (parallelStart1) at ($(p1) + (90-\n1:0.075cm)$)
      coordinate (parallelEnd1) at ($(p2) + (90-\n1:0.075cm)$)
      coordinate (parallelStart2) at ($(p1) - (90-\n1:0.075cm)$)
      coordinate (parallelEnd2) at ($(p2) - (90-\n1:0.075cm)$);

    \path let \p1 = ($(p4)-(p3)$), \n1 = {atan2(\y1,\x1)} in
      coordinate (q1) at ($(p3) + (90-\n1:0.0745cm)$)
      coordinate (q2) at ($(p4) + (90-\n1:0.0745cm)$)
      coordinate (s1) at ($(p3) - (90-\n1:0.074cm)$)
      coordinate (s2) at ($(p4) - (90-\n1:0.074cm)$);


    \draw (parallelStart1) -- (parallelEnd1);
    \draw (parallelStart2) -- (parallelEnd2);
    \draw (p1) --  (p2);
    
    \draw (p2) -- (p3);

    \draw (q1) --  (q2);
    \draw (p3) --  (p4);
    \draw (s1) --  (s2);

    \draw (p4) -- (p5);
    \draw[dashed]  (p5) -- node [right] {$c_2$} (p6);
    \draw[dashed] (p6) -- node [below] {$b_2$} (p7);
    \draw[dashed] (p7) -- node [left] {$a_2$} (p1);

    \foreach \i in {1,...,7} {
      \draw[fill=white] ({360/7*(\i - 3)}:2) circle (2pt);}
  \end{scope}
\end{tikzpicture}
    \caption{The Coxeter diagrams of the Coxeter polyhedra $P_1^4$ (left) and $P_2^4$ (right). Here, $a_1 = \sqrt{\frac{\sqrt{5}}{2\sqrt{5}-3}}$, $b_1 = \frac{\sqrt{3+\sqrt{5}}}{\sqrt{13-5\sqrt{5}}}$, $c_1 = \sqrt{\frac{\sqrt{5}}{\sqrt{5}-1}}$, $a_2 = \frac{\sqrt{2}\sqrt{\sqrt{5}+7}}{\sqrt{11}}$, $b_2 = \frac{3+\sqrt{5}}{\sqrt{2(71-29\sqrt{5})}}$, and $c_2 = \frac{2\sqrt{\sqrt{5}+9}}{\sqrt{38}}$. 
    These values can be found as follows. The condition that the Gram matrix of $P_i^4$ has signature $(4,1)$ implies that the principal corner minors of order $6$ and the determinant of the Gram matrix itself are equal to $0$. Therefore we obtain a system of equations with respect to $a_i$, $b_i$, $c_i$. Our computations coincide with those presented in \cite[Example~3.16]{dotti}. One can also check that with these values $a_i$, $b_i$, $c_i$, the Gram matrix of $P_i^4$ indeed has signature $(4,1)$, and thus by \cite[Theorem 2.1]{Vinberg_polytopes} the acute-angled polyhedron $P_i^4$ indeed exists and is uniquely determined up to an isometry of $\mathbb{H}^4$.
    }
    \label{fig:p12-4}
\end{figure}


\subsubsection{Arithmetic invariants}

Computing Gram matrices, we see that $S_1^4$ and $P_1^4$ (and $\overline{S_1^4}$) all share the same Vinberg field $k = \QQ(\sqrt 5)$. Since $\Gamma_{S_1^4}$ is Zariski-dense in $\PO(4,1)$, the ambient group of $\Gamma_{S_1^4}$ is the same as that of $\Gamma_{P_1^4}$.

The Gram matrix of $S_1^4$ is given by 
\[
G(S_1^4) = \left(\begin{array}{rrrrr}
1 & -\sqrt 2/2 & 0 & 0 & 0 \\
-\sqrt 2/2 & 1 & -\frac{1}{2} & 0 & 0 \\
0 & -\frac{1}{2} & 1 & -\frac{1}{2} a & 0 \\
0 & 0 & -\frac{1}{2} a & 1 & -\frac{1}{2} \\
0 & 0 & 0 & -\frac{1}{2} & 1
\end{array}\right), 
\]
where $a = 2\cos(\pi/5) = \tfrac{\sqrt 5 + 1}2$. The subgroup of $\mathrm{GL}_5(\mathbb{R})$ generated by the reflections $\gamma_i : v \mapsto v - 2 (v^T G(S_1^4) e_i)e_i$ for $i=1, \ldots, 5$ preserves $G(S_1^4)$, where the latter is viewed as a quadratic form on $\mathbb{R}^5$, and the image of this group in $\mathrm{PGL}_5(\mathbb{R})$ is conjugate to $\Gamma_{S_1^4}$. In the new basis for $\mathbb{R}^5$ obtained by multiplying $e_1$ by $\sqrt 2$, the matrix for the form $G(S_1^4)$ is given by 
\[
Q_1^4 = \left(\begin{array}{rrrrr}
2 & -1 & 0 & 0 & 0 \\
-1 & 1 & -\frac{1}{2} & 0 & 0 \\
0 & -\frac{1}{2} & 1 & -\frac{1}{2} a & 0 \\
0 & 0 & -\frac{1}{2} a & 1 & -\frac{1}{2} \\
0 & 0 & 0 & -\frac{1}{2} & 1
\end{array}\right)
\]
and the entries of the $\gamma_i$ lie in $k$, so that the ambient group of $\Gamma_{S_1^4}$ is isomorphic to $\PO(Q_1^4)$.  The latter group is admissible for $\PO(4, 1)$ and so $\Gamma_{P_1^4}$ is quasi-arithmetic (in fact, we have checked that $\Gamma_{P_1^4}$ is properly quasi-arithmetic; this was indeed already observed by Dotti \cite[Example~3.16]{dotti}). 

\medskip

Similarly, the Vinberg fields of $S_2^4$ and $P_2^4$ are both equal to $k$, and the ambient group of $\Gamma_{P_2^4}$ is the same as that of $\Gamma_{S_2^4}$. The Gram matrix of $S_2^4$ is
\[
Q_2^4 = \left(\begin{array}{rrrrr}
1 & -\frac{1}{2} a & 0 & 0 & 0 \\
-\frac{1}{2} a & 1 & -\frac{1}{2} & 0 & 0 \\
0 & -\frac{1}{2} & 1 & -\frac{1}{2} a & 0 \\
0 & 0 & -\frac{1}{2} a & 1 & -\frac{1}{2} \\
0 & 0 & 0 & -\frac{1}{2} & 1
\end{array}\right),
\]
so the ambient group for $\Gamma_{S_2^4}$ is $\PO(Q_2^4)$. It follows that $\Gamma_{P_2^4}$ is  quasi-arithmetic (and again, one can check, properly quasi-arithmetic). 

\medskip

Now we prove that the ambient groups $\PO(Q_1^4)$ and $\PO(Q_2^4)$ are not $k$-isogenous using Lemma \ref{isogeny}. This was already established by Dotti \cite[Example~3.16]{dotti}, but we include the argument here for the convenience of the reader. A straightforward computation in SAGE \cite{sagemath} shows that the Hasse invariants of $Q_1^4$ and $Q_2^4$ differ at the prime $\mathfrak p_5 = (1-2a)$, i.e., the ideal generated by a square root of 5 in $k$. Recall that the Hasse invariant of a quadratic form $Q$ over a field $K$ is defined to be 
\begin{equation} \label{defn_hasse}
h_K(Q) = \prod_{i < j} (a_i, a_j)_K
\end{equation}
where $Q$ is isometric to a diagonal form with coefficients $a_i$, and $(x, y)_K$ is the Hilbert symbol over $K$ (which equals $1$ if the corresponding quaternion algebra is $K$-split and $-1$ otherwise, see e.g. \cite{McR} for the definition). 

Assume that $Q_1^4$ is isometric over $k$ to $\lambda Q_2^4$ for some $\lambda\in k^\times$. Since $\tfrac{\det(Q_1^4)}{\det(Q_2^4)} = \frac{6-2a}5$ and the factorisation of $(6-2a)$ into prime ideals is $(2) \cdot \mathfrak p_5$, we see that $\lambda = 2u\mathfrak p_5$ modulo squares, where $u$ is a unit in the integer ring $\ZZ[a]$.  

Let $k_5$ be the localisation of $k$ at $\mathfrak p_5$. By Theorem 2.6.6(3) in \cite{McR}, we have that $(\lambda, \lambda)_{k_5} = 1$ if and only if $-1$ is a square modulo $\mathfrak p_5$. This is indeed the case since $\ZZ[a]/\mathfrak p_5 = \mathbb F_5$. So the Hasse invariant of $\lambda Q_2^4$ over $k_5$ is equal to that of $Q_2^4$ and hence distinct from that of $Q_1^4$, so that the forms $Q_1^4$ and $Q_2^4$ cannot be isometric over $k_5$, a fortiori not over $k$.


\subsection{In five dimensions} \label{5d_Makarov}

\subsubsection{Description} 
Each of the Coxeter diagrams
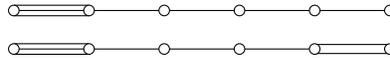
\begin{center}
\begin{tikzpicture}
  \coordinate (1) at (0, 0) ;
  \coordinate (2) at (1, 0) ;
  \coordinate (3) at (2, 0) ;
  \coordinate (4) at (3, 0) ;
  \coordinate (5) at (4, 0) ;
  \coordinate (6) at (5, 0) ;

  \draw (1) -- (2) ;
  \draw (0,0.07) -- (1,0.07) ; 
  \draw (0,-0.07) -- (1,-0.07) ;
  \draw (2) -- (3) ; 
  \draw (3) -- (4) ; 
  \draw (4) -- (5) ; 
  \draw (5) -- (6) ;

  \fill[white] (1) circle (2pt) ;
  \fill[white] (2) circle (2pt) ; 
  \fill[white] (3) circle (2pt) ;
  \fill[white] (4) circle (2pt) ;
  \fill[white] (5) circle (2pt) ;
  \fill[white] (6) circle (2pt) ;
  
  \draw (1) circle (2pt) ;
  \draw (2) circle (2pt) ; 
  \draw (3) circle (2pt) ;
  \draw (4) circle (2pt) ;
  \draw (5) circle (2pt) ; 
  \draw (6) circle (2pt) ; 
\end{tikzpicture} 
\end{center}

\begin{center}
\begin{tikzpicture}
  \coordinate (1) at (0, 0) ;
  \coordinate (2) at (1, 0) ;
  \coordinate (3) at (2, 0) ;
  \coordinate (4) at (3, 0) ;
  \coordinate (5) at (4, 0) ;
  \coordinate (6) at (5, 0) ;

  \draw (1) -- (2) ;
  \draw (0,0.07) -- (1,0.07) ; 
  \draw (0,-0.07) -- (1,-0.07) ;
  \draw (2) -- (3) ; 
  \draw (3) -- (4) ; 
  \draw (4) -- (5) ; 

  \draw (4,0.06) -- (5,0.06) ; 
  \draw (4,-0.06) -- (5,-0.06) ;

  \fill[white] (1) circle (2pt) ;
  \fill[white] (2) circle (2pt) ; 
  \fill[white] (3) circle (2pt) ;
  \fill[white] (4) circle (2pt) ;
  \fill[white] (5) circle (2pt) ;
  \fill[white] (6) circle (2pt) ;
  
  \draw (1) circle (2pt) ;
  \draw (2) circle (2pt) ; 
  \draw (3) circle (2pt) ;
  \draw (4) circle (2pt) ;
  \draw (5) circle (2pt) ; 
  \draw (6) circle (2pt) ; 
\end{tikzpicture}
\end{center}

\noindent gives a simplex in $\mathbb{P}(\RR^{5,1})$ with a single hyperideal vertex with link corresponding to the diagram
\begin{center}
\begin{tikzpicture}
  \coordinate (1) at (0, 0) ;
  \coordinate (2) at (1, 0) ;
  \coordinate (3) at (2, 0) ;
  \coordinate (4) at (3, 0) ;
  \coordinate (5) at (4, 0) ;

  \draw (1) -- (2) ;
  \draw (0,0.07) -- (1,0.07) ; 
  \draw (0,-0.07) -- (1,-0.07) ;
  \draw (2) -- (3) ; 
  \draw (3) -- (4) ; 
  \draw (4) -- (5) ;

  \fill[white] (1) circle (2pt) ;
  \fill[white] (2) circle (2pt) ; 
  \fill[white] (3) circle (2pt) ;
  \fill[white] (4) circle (2pt) ;
  \fill[white] (5) circle (2pt) ;
  
  \draw (1) circle (2pt) ;
  \draw (2) circle (2pt) ; 
  \draw (3) circle (2pt) ;
  \draw (4) circle (2pt) ;
  \draw (5) circle (2pt) ; 
\end{tikzpicture} 
\end{center}

\noindent which represents a compact simplex in $\mathbb{H}^4$. Truncating these vertices yields compact Coxeter polyhedra $P_1^5, P \subset \mathbb{H}^5$ (see Figure~\ref{p12-5} for their Coxeter diagrams), respectively, each possessing a facet orthogonal to all adjacent facets and which is isometric to the polyhedron with the latter Coxeter diagram.
Moreover, there is a nonadjacent facet $F$ of $P$ forming even angles (of $\frac{\pi}{2}$ or $\frac{\pi}{4}$) with all remaining facets. Let $P_2^5$ be the double of $P$ along $F$. 

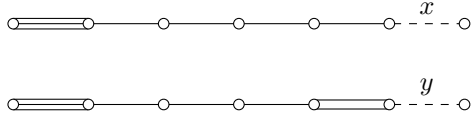
\begin{figure}
\begin{center}
\begin{tikzpicture}
  \coordinate (1) at (0, 0) ;
  \coordinate (2) at (1, 0) ;
  \coordinate (3) at (2, 0) ;
  \coordinate (4) at (3, 0) ;
  \coordinate (5) at (4, 0) ;
  \coordinate (6) at (5, 0) ;
  \coordinate (7) at (6, 0) ;

  \draw (1) -- (2) ;
  \draw (0,0.07) -- (1,0.07) ; 
  \draw (0,-0.07) -- (1,-0.07) ;

  \draw (2) -- (3) ; 
  \draw (3) -- (4) ; 
  \draw (4) -- (5) ; 
  \draw (5) -- (6) ; 
  \draw[dashed] (7) --  node [above] {$x$} (6) ; 

  \fill[white] (1) circle (2pt) ;
  \fill[white] (2) circle (2pt) ; 
  \fill[white] (3) circle (2pt) ;
  \fill[white] (4) circle (2pt) ;
  \fill[white] (5) circle (2pt) ;
  \fill[white] (6) circle (2pt) ;
  \fill[white] (7) circle (2pt) ;
  
  \draw (1) circle (2pt) ;
  \draw (2) circle (2pt) ; 
  \draw (3) circle (2pt) ;
  \draw (4) circle (2pt) ;
  \draw (5) circle (2pt) ; 
  \draw (6) circle (2pt) ; 
  \draw (7) circle (2pt) ; 
\end{tikzpicture} 

\vskip 0.5cm
\begin{tikzpicture}
  \coordinate (1) at (0, 0) ;
  \coordinate (2) at (1, 0) ;
  \coordinate (3) at (2, 0) ;
  \coordinate (4) at (3, 0) ;
  \coordinate (5) at (4, 0) ;
  \coordinate (6) at (5, 0) ;
  \coordinate (7) at (6, 0) ;

  \draw (1) -- (2) ;
  \draw (0,0.07) -- (1,0.07) ; 
  \draw (0,-0.07) -- (1,-0.07) ; 
  
  \draw (2) -- (3) ; 
  \draw (3) -- (4) ; 
  \draw (4) -- (5) ; 
  \draw (4,0.06) -- (5,0.06) ; 
  \draw (4,-0.06) -- (5,-0.06) ;  
  \draw[dashed] (7) --  node [above] {$y$} (6) ; 

  \fill[white] (1) circle (2pt) ;
  \fill[white] (2) circle (2pt) ; 
  \fill[white] (3) circle (2pt) ;
  \fill[white] (4) circle (2pt) ;
  \fill[white] (5) circle (2pt) ;
  \fill[white] (6) circle (2pt) ;
  \fill[white] (7) circle (2pt) ;
  
  \draw (1) circle (2pt) ;
  \draw (2) circle (2pt) ; 
  \draw (3) circle (2pt) ;
  \draw (4) circle (2pt) ;
  \draw (5) circle (2pt) ; 
  \draw (6) circle (2pt) ; 
  \draw (7) circle (2pt) ; 
\end{tikzpicture}
\caption{The Coxeter diagrams of the Coxeter polyhedra $P_1^5$ and $P$. Here, $x = \frac{1}{2} \sqrt{\frac{1}{2} (7 + \sqrt{5})}$ and $y = \frac{\sqrt{3 + \sqrt{5}}}{2}$. These values were obtained in precisely the same way as the values in Figure~\ref{fig:p12-4}.}
\label{p12-5}
\end{center}
\end{figure}

\subsubsection{Arithmetic invariants} As in the $4$-dimensional case, the Vinberg field of $S_i^5$ (which coincides with the Vinberg field of $P_i^5$) is $k = \mathbb{Q}(\sqrt{5})$ for $i=1,2$. The Gram matrix of $S_1^5$ is
\[
Q_1^5 = \begin{pmatrix} 
1 & -\frac{1}{2}a & 0 & 0 & 0 & 0 \\
-\frac{1}{2}a & 1 & -\frac{1}{2} & 0 & 0 & 0 \\
0 & -\frac{1}{2} & 1 & -\frac{1}{2} & 0 & 0 \\
0 & 0 & -\frac{1}{2} & 1 & -\frac{1}{2} & 0 \\
0 & 0 & 0 & -\frac{1}{2} & 1 & -\frac{1}{2} \\
0 & 0 & 0 & 0 & -\frac{1}{2} & 1
\end{pmatrix},
\]
so the ambient group for $\Gamma_{S_1^5}$ is $\PO(Q_1^5)$. Since this group is admissible for $\PO(5,1)$, we have that $\Gamma_{P_1^5}$ is quasi-arithmetic. 

The Gram matrix of $S_2^5$ is
\[
 \begin{pmatrix} 
1 & -\frac{1}{2}a & 0 & 0 & 0 & 0 \\
-\frac{1}{2}a & 1 & -\frac{1}{2} & 0 & 0 & 0 \\
0 & -\frac{1}{2} & 1 & -\frac{1}{2} & 0 & 0 \\
0 & 0 & -\frac{1}{2} & 1 & -\frac{1}{2} & 0 \\
0 & 0 & 0 & -\frac{1}{2} & 1 & -\sqrt{2}/2 \\
0 & 0 & 0 & 0 & -\sqrt{2}/2 & 1
\end{pmatrix}.
\]
By multiplying the last basis vector by $\sqrt{2}$ we obtain the equivalent matrix
\[
Q_2^5= \begin{pmatrix} 
1 & -\frac{1}{2}a & 0 & 0 & 0 & 0 \\
-\frac{1}{2}a & 1 & -\frac{1}{2} & 0 & 0 & 0 \\
0 & -\frac{1}{2} & 1 & -\frac{1}{2} & 0 & 0 \\
0 & 0 & -\frac{1}{2} & 1 & -\frac{1}{2} & 0 \\
0 & 0 & 0 & -\frac{1}{2} & 1 & -1 \\
0 & 0 & 0 & 0 & -1 & 2
\end{pmatrix},
\]
so the ambient group of $\Gamma_{S_2}^5$ is $\PO(Q_2^5)$. This group is admissible for $\PO(5,1)$, and so $\Gamma_{P_2^5}$ is also quasi-arithmetic. Though we will not need this, it is easy to verify using Vinberg's arithmeticity criterion \cite[Thm.~2]{Vin67} that the polyhedra $P_1^5$ and $P$ (hence also~$P_2^5$) are moreover arithmetic.

By Lemma \ref{isogeny}, the $k$-groups $\PO(Q_1^5)$ and $\PO(Q_2^5)$ are $k$-isogenous only if $\frac{\det(Q_1^5)}{\det(Q_2^5)} = \frac{8-2a}{4}$ is a square in $k$. The latter is true if and only if $8-2a$ is a square in $k$, hence in its integer ring $\mathbb{Z}[a]$. We have $8-2a = 2(4-a)$, and the norm of $(4-a)$ is 11. So the norm of $(8-2a)$ is $44$, which is not a square in $\ZZ$, so that $8-2a$ is not a square in $\ZZ[a]$.


\section{Coxeter polyhedra and commensurability} \label{proofs}

\subsection{Garlands} \label{garlands}

Let $P_1, P_2$ be two compact Coxeter polyhedra in $\HH^d$ satisfying the following hypothesis:

\begin{assumption} \label{2side_gluing}
  Each of the $P_i$ contains two nonadjacent facets isometric to the same Coxeter polyhedron $R$ in $\HH^{d-1}$ and orthogonal to all adjacent facets.
\end{assumption}

Given any sequence $\alpha \in \{1, 2\}^n$ we can form a compact polyhedron $P_\alpha$ by gluing copies of $P_1, P_2$ according to $\alpha$ (what Vinberg \cite{Vinberg_polytopes} calls a {\em garland}). Formally, we choose an indexing of the two $R$-facets of $P_i$ which we denote by $\pl^\pm P_i$ respectively, and then we obtain $P_\alpha$ by identifying each $\pl^+ P_{\alpha_i}$ with $\pl^- P_{\alpha_{i+1}}$ for $1\le i \le n-1$. As a consequence of Assumption \ref{2side_gluing}, the resulting polyhedron remains a Coxeter polyhedron. 

\medskip

We can also perform a more restricted construction under a weaker hypothesis. 

\begin{assumption} \label{1side_gluing}
  The polyhedron $P_1$ contains a facet isometric to a Coxeter polyhedron $R$ in $\HH^{d-1}$ and orthogonal to all adjacent facets, while $P_2$ contains two nonadjacent facets isometric to $R$  and orthogonal to all adjacent facets. 
\end{assumption}

Under this assumption we can form for each $n \ge 1$ a Coxeter polyhedron $L_n$ by gluing $n$ copies of $P_2$ and then capping one end with a copy of $P_1$. 


\subsection{Infinitely many commensurability classes}

\begin{prop} \label{commclass_garlands}
  Suppose $P_1, P_2$ are two compact quasi-arithmetic Coxeter polyhedra in $\HH^d$, $d \geq 3$, such that the $\Gamma_{P_i}$ have distinct ambient groups. Suppose that $P_1, P_2$ satisfy Assumption \ref{1side_gluing}. Then the $L_n$, $n \geq 1$, are pairwise incommensurable. 
\end{prop}

\begin{proof}
Suppose there is a closed hyperbolic orbifold $M$ with covers $\pi : M \to L_n$, $\pi' : M \to L_{n'}$. Let $L_n^1$ (resp., $L_{n'}^1$) be the copy of $P_1$ in $L_n$ (resp., $L_{n'}$), thought of as a compact hyperbolic orbifold with connected totally geodesic boundary. By Proposition \ref{qa_hybrid} and Lemma 3.2 in \cite{croissmax}\footnote{The latter is stated only for manifolds in this reference, but the proof works as written for orbifolds.}, we have $\pi^{-1}(L_n^1) = \pi'^{-1}(L_{n'}^1)$. Since $L_n^1$ and $L_{n'}^1$ share the same volume, namely, that of $P_1$, it follows that $\pi$ and $\pi'$ share the same degree. But then $L_n$ and $L_{n'}$ share the same volume, and so $n=n'$.
\end{proof}


\subsection{Quantitative estimate for the number of commensurability classes}

If $\alpha \in \{1, 2\}^n$ we denote by $\bar\alpha$ its mirror sequence (that is, $\bar\alpha_i = \alpha_{n+1-i}$). The following proposition is similar to \cite[Proposition 2.1]{croissmax}, and we will deduce it from this result.  

\begin{prop} \label{commclass_garlands2}
  If $P_1, P_2 \subset \mathbb{H}^d$, $d \geq 3$, satisfying Assumption \ref{2side_gluing} are quasi-arithmetic with distinct ambient groups, then for any $n \ge 1$ and any sequences $\alpha , \beta \in \{1, 2\}^n$ such that $\Gamma_{P_\alpha}$ is commensurable with $\Gamma_{P_\beta}$, we have that $\beta$ or $\bar\beta$ is a subsequence of $(\alpha, \bar{\alpha})$. 
\end{prop}

An immediate consequence is the following more precise version of Proposition~\ref{commclass_garlands}. 

\begin{cor} \label{counting}
  There are at least $\tfrac{2^n}{2n}$ distinct commensurability classes among those of the $P_\alpha$ for $\alpha \in \{1, 2\}^n$. 
\end{cor}

\begin{proof}[Proof of Proposition \ref{commclass_garlands2}]
  Let $n \ge 1$ and $\alpha \in \{1, 2\}^n$. Let $\sigma_1, \sigma_n$ be the reflections in the facets $\pl^- P_{\alpha_1}$ and $\pl^+ P_{\alpha_n}$. Since these facets are nonadjacent and orthogonal to all adjacent facets, there is a surjective morphism $\varphi$ from $\Gamma_{P_\alpha}$ to the infinite dihedral group $\langle \sigma_1, \sigma_n\rangle$. Let $\eps$ be the morphism $\langle \sigma_1, \sigma_n\rangle \to \ZZ/2\ZZ$ sending both $\sigma_1$ and $\sigma_n$ to the generator. We define $M_\alpha$ be the double orbifold cover of $P_\alpha$ corresponding to the morphism $\eps \circ \varphi$. 

  Topologically, the orbifold $M_\alpha$ is obtained by first viewing $P_\alpha$ as an orbifold with totally geodesic boundary $\pl^- P_{\alpha_1}\cup \pl^+ P_{\alpha_n}$ (so that the orbifold fundamental group of $P_\alpha$ is generated by the reflections in all facets of $P_\alpha$ apart from $\pl^- P_{\alpha_1}$ and $\pl^+ P_{\alpha_n}$), and then doubling along $\pl^- P_{\alpha_1}\cup \pl^+ P_{\alpha_n}$. This is the setting of \cite{croissmax}, so if some $P_\beta$ is commensurable with $P_\alpha$ (so that $M_\alpha, M_\beta$ are commensurable to each other as well) we can apply Proposition 2.1 in loc. cit. to conclude that the sequences $(\alpha,\bar \alpha)$  and $(\beta, \bar\beta)$ are cyclic permutations of one another (these sequences are both palindromic so we need not apply a symmetry). In particular $\beta$ or $\bar\beta$ is a subsequence of $(\alpha, \bar\alpha)$. 
\end{proof}

\begin{remark}
  Under the conditions of Proposition \ref{commclass_garlands2}, we may in fact conclude that~$\beta$ on the nose is a subsequence of $(\alpha, \bar{\alpha})$, so that it is even true that the $P_\alpha$ represent at least $\frac{2^n}{n}$ distinct commensurability classes as $\alpha$ varies within $\{1,2\}^n$. Indeed, if $(\alpha, \bar\alpha)$ and $(\beta, \bar\beta)$ are cyclic permutations of one another and $\beta$ is neither $\alpha$ nor $\bar \alpha$, then $(\alpha, \bar{\alpha})$ is invariant under some nontrivial dihedral group of permutations, and in particular, under some (possibly trivial) cyclic permutation shifting $\beta$ into $(\alpha, \bar\alpha)$. 
\end{remark}


\subsection{Proof of the main results}

In the notation of Section \ref{4d_Makarov}, the polyhedra $P_1 = P_1^4$ and $P_2 = P_2^4$ each have two nonadjacent facets isometric to $R= T_2$ and orthogonal to all adjacent facets, and hence satisfy  Assumption \ref{2side_gluing}. Since $\Gamma_{P_1^4}$ and $\Gamma_{P_2^4}$ are both quasi-arithmetic with distinct ambient groups, the conditions of Proposition \ref{commclass_garlands2} are satisfied, and Theorem \ref{quantitative} follows from Corollary \ref{counting}.

In the notation of Section \ref{5d_Makarov}, we can take $P_1 = P_1^5$ and $P_2 = P_2^5$. Then $P_1, P_2$ satisfy Assumption \ref{1side_gluing}, and the other assumptions of Proposition \ref{commclass_garlands}. This proves the case $d=5$ of Theorem \ref{Main}.


\subsection{Proof of Theorem \ref{non-quasi-arit}}

The proof is similar to that of Corollary \ref{qa_hybrid}. Let $i=4,5$ and let $P$ be a garland containing copies of both $P_1^i$ and $P_2^i$. Assume that $\Gamma_P$ is quasi-arithmetic; then there exists a totally real $k$-group $\G$ which is admissible for $\PO(d, 1)$ and such that $\Gamma_P \subset \G(k)$. Since $\Gamma_P$ contains Zariski-dense subgroups of both $\Gamma_{P_j^i}$, the field $k$ contains $\QQ(\sqrt 5)$ and there are $k$-isogenies $\G_j \to \G$ for $j=1, 2$. Since $\G$ is admissible, we have that $k=\QQ(\sqrt 5)$, and it follows that $\G, \G_1, \G_2$ are all $k$-isogenous, which is not the case.


\bibliographystyle{siam}
\bibliography{bib}

\end{document}

%% file: macros.tex
\DeclareFontFamily{U}{wncy}{}
\DeclareFontShape{U}{wncy}{m}{n}{<->wncyr10}{}
\DeclareSymbolFont{mcy}{U}{wncy}{m}{n}
\DeclareMathSymbol{\Sha}{\mathord}{mcy}{"58}

\newcommand{\eps}{\varepsilon}

\newcommand{\pl}{\partial}

\newcommand{\tr}{\operatorname{tr}}

\newcommand{\ad}{\operatorname{Ad}}

\newcommand{\interval}[4]{
  \ifthenelse{ \equal{#1}{o} } {\mathopen{]}} {\mathopen{[}}
  #2, #3
  \ifthenelse{ \equal{#4}{o} } {\mathclose{[}} {\mathclose{]}}
}

\newcommand{\G}{\mathbf{G}}
\newcommand{\RR}{\mathbb R}
\newcommand{\ZZ}{\mathbb Z}

\newcommand{\QQ}{\mathbb Q}